\documentclass[11pt]{amsart}

\usepackage{amsmath,amssymb,latexsym,soul,cite,mathrsfs}

\usepackage{color,enumitem,graphicx}
\usepackage[colorlinks=true,urlcolor=blue,
citecolor=red,linkcolor=red,linktocpage,pdfpagelabels,
bookmarksnumbered,bookmarksopen]{hyperref}
\usepackage[english]{babel}

\usepackage[left=2.9cm,right=2.9cm,top=2.8cm,bottom=2.8cm]{geometry}
\usepackage[hyperpageref]{backref}

\usepackage[colorinlistoftodos]{todonotes}
\makeatletter
\providecommand\@dotsep{5}
\def\listtodoname{List of Todos}
\def\listoftodos{\@starttoc{tdo}\listtodoname}
\makeatother

\numberwithin{equation}{section}

\newcommand{\e}{\epsilon}

\newcommand{\Om} {\Omega}

\newcommand{\la} {\lambda}
\newcommand{\La} {\Lambda}
\newcommand{\noi} {\noindent}

\newcommand{\mb} {\mathbb}
\newcommand{\mc} {\mathcal}

\usepackage[all]{xy}
\catcode`\@=11
\def\theequation{\@arabic{\c@section}.\@arabic{\c@equation}}
\catcode`\@=12

\newtheorem{Theorem}{Theorem}[section]
\newtheorem{Lemma}[Theorem]{Lemma}
\newtheorem{Proposition}[Theorem]{Proposition}
\newtheorem{Corollary}[Theorem]{Corollary}
\newtheorem{Remark}[Theorem]{Remark}

\begin{document}

{\vspace{0.01in}}

\title
{ \sc On an Anisotropic $p$-Laplace equation with variable singular exponent}
\author{Kaushik Bal, Prashanta Garain and Tuhina Mukherjee}

\maketitle

\begin{abstract}
{In this article, we study the following anisotropic p-Laplacian equation with variable exponent given by 
\begin{equation*}
(P)\left\{\begin{split}
-\Delta_{H,p}u&=\frac{\la f(x)}{u^{q(x)}}+g(u)\text{ in }\Omega,\\
u&>0\text{ in }\Omega,\,u=0\text{ on }\partial\Omega,
\end{split}\right.
\end{equation*}
under the assumption $\Omega$ is a bounded smooth domain in $\mathbb{R}^N$ with $p,N\geq 2$, $\la>0$ and $0<q \in C(\bar \Om)$. For the purely singular case that is  $g\equiv 0$, we proved existence and uniqueness of solution. We also demonstrate the existence of multiple solution to $(P)$ provided $f\equiv 1$ and $g(u)=u^r$ for $r\in (p-1,p^*-1)$.}

\medskip

\noi {Key words: Anisotropic $p$-Laplace equation, Singular nonlinearity, Variable exponent, Approximation method, Variational approach.}

\medskip

\noi \textit{2010 Mathematics Subject Classification:} 35D30, 35J62, 35J75.
\end{abstract}

\section{Introduction and Main results}
 In this article, we establish the existence of unique  weak solution for the problem
\begin{equation}\label{maineqn}
\left\{\begin{split}
-\Delta_{H,p}u&=\frac{f(x)}{u^{q(x)}}\text{ in }\Omega,\\
u&>0\text{ in }\Omega,\,u=0\text{ on }\partial\Omega,
\end{split}\right.
\end{equation}
and existence of at least two weak solutions for
\begin{equation}\label{maineqn1}
\left\{\begin{split}
-\Delta_{H,p}u&=\frac{\la }{u^{q(x)}}+u^r\text{ in }\Omega,\\
u&>0\text{ in }\Omega,\,u=0\text{ on }\partial\Omega,
\end{split}\right.
\end{equation}
 where $\Omega$ is a bounded smooth domain in $\mathbb{R}^N$ with $N\geq 2$ and $0<q \in C(\bar \Om)$. Our other important assumptions for \eqref{maineqn} are $2\leq p<\infty$ and $f\in L^1(\Omega)\setminus\{0\}$ is a nonnegative function (see Theorem \ref{thm1}-\ref{thm1new} for precise assumptions). For the problem \eqref{maineqn1}, we assume the parameter $\lambda$ to be positive and $r\in(p-1,p^*-1)$, where $p^*=\frac{Np}{N-p}$ is the critical Sobolev exponent for $1<p<N$ (refer Theorem \ref{thm3}-\ref{thm3new}). The key operator in our problems, that is the Finsler p-Laplacian is defined as
 $$
 \Delta_{H,p}u:=\text{div}\big(H(\nabla u)^{p-1}\nabla_{\eta}H(\nabla u)\big),
 $$ where $\nabla_{\eta}$ denotes the gradient operator with respect to the $\eta$ variable and $H:\mathbb{R}^N\to[0,\infty)$ is known as the Finsler-Minkowski norm satisfying the following hypothesis:
\begin{enumerate}
\item[(H0)] $H(x)\geq 0$, for every $x\in\mathbb{R}^N$.
\item[(H1)] $H(x)=0$, if and only if $x=0$.
\item[(H2)] $H(tx)=|t|H(x)$, for every $x\in\mathbb{R}^N$ and $t\in\mathbb{R}$.
\item[(H3)] $H\in C^{\infty}\left(\mathbb{R}^N\setminus\{0\}\right)$.
\item[(H4)] the Hessian matrix $\nabla_{\eta}^2\big(\frac{H^2}{2}\big)(x)$ is positive definite for all $x\in\mathbb{R}^N\setminus\{0\}$.
\end{enumerate}
The dual $H_0:\mathbb{R}^N\to[0,\infty)$ of $H$ is defined by 
\begin{equation}\label{dual}
H_0(\xi):=\sup_{x\in\mathbb{R}^N\setminus\{0\}}\frac{\langle x,\xi\rangle}{H(x)}.
\end{equation}
More details on this can be found in Bao-Chern-Shen \cite{BCS}, Xia \cite{Xiathesis} and Rockafellar \cite{Rbook}.
 
Anisotropic $p$-Laplace equations has been a topic of considerable attention in the recent years. We refer to Alvino-Ferone-Trombetti-Lions \cite{AFTL}, Belloni-Ferone-Kawohl \cite{BFKzamp}, Belloni-Kawohl-Juutinen \cite{BKJ}, Bianchini-Giulio \cite{BC}, Xia \cite{Xiathesis},  Cianchi-Salani\cite{CS}, Ferone-Kawohl \cite{FK}, Kawohl-Novaga \cite{KN} and the references therein, for detailed analysis of the operator and its important features.
We present some remarks and examples now to give a little more insight on the anisotropic p-Laplace operator and the adjoined norm. 
\begin{Remark}\label{hypormk2}
Since all norms in $\mathbb{R}^N$ are equivalent, there exist positive constants $C_1, C_2$ such that
$$
C_1|x|\leq H(x)\leq C_2|x|,\quad\forall \,x\in\mathbb{R}^N.
$$
We are able to define an equivalent norm on the Sobolev space $W_0^{1,p}(\Om)$ as 
\begin{equation}\label{equinorm}
\|u\|_{X}:=\left(\int_{\Om}H(\nabla u)^p\,dx\right)^{\frac{1}{p}}.
\end{equation}
\end{Remark}
\noi \textbf{Examples:} Let $x=(x_1,x_2,\ldots,x_N)\in\mathbb{R}^N$,
\begin{enumerate}
\item[(i)] then, for $t>1$, we define 
\begin{equation}\label{ex1}
H_t(x):=\Big(\sum_{i=1}^{N}|x_i|^t\Big)^\frac{1}{t};
\end{equation}
\item[(ii)] for $\lambda,\mu>0$, we define
\begin{equation}\label{ex2}
H_{\lambda,\mu}(x):=\sqrt{\lambda\sqrt{\sum_{i=1}^{N}x_i^{4}}+\mu\sum_{i=1}^{N}x_i^{2}}.
\end{equation}
\end{enumerate}
Then, one may check that the functions $H_t, H_{\lambda,\mu}:\mathbb{R}^N\to[0,\infty)$ given by \eqref{ex1} and \eqref{ex2} satisfies all the hypothesis from $(H0)-(H4)$ or refer Mezei-Vas \cite{MV}.

\begin{Remark}\label{exrmk1}
For $i=1,2$ if $\lambda_i,\mu_i$ are positive real numbers such that $\frac{\lambda_1}{\mu_1}\neq\frac{\lambda_2}{\mu_2}$, then $H_{\lambda_1,\mu_1}$ and $H_{\lambda_2,\mu_2}$ given by \eqref{ex2} defines two non-isometric norms in $\mathbb{R}^N$.
\end{Remark}

\begin{Remark}\label{exrmk2}
Moreover, for $H=H_t$ given by \eqref{ex1} we have
\begin{equation}\label{ex}
\Delta_{H_t,p}u=
\begin{cases}
\Delta_p u:=\text{div}(|\nabla u|^{p-2}\nabla u),\, (p\text{-Laplacian})\,\text{if }t=2,\,1<p<\infty,\\
\mathcal{S}_p=\sum_{i=1}^{N}\frac{\partial}{\partial x_i}\Big(|u_i|^{p-2}u_i\Big),\,\text{if }t=p\in(1,\infty),
\end{cases}
\end{equation}
where $u_i:=\frac{\partial u}{\partial x_i}$, for $i=1,2,\ldots,N$.
\end{Remark}
This last remark puts an important touch to our problem stating that Finsler p-Laplacian generalizes the $p$-Laplacian. So our result are true for a general set of quasilinear operators and we strongly point out here that \eqref{maineqn} and \eqref{maineqn1} are new even for Laplacian.

 The equations \eqref{maineqn} and \eqref{maineqn1} are singular in the sense that the nonlinearities in our consideration blow up near the origin, due to the presence of the term $u^{-q(x)}$ where $q$ is a positive function. The fundamental feature of our article is the singular variable exponent $q:\overline{\Om}\to(0,\infty)$ which is a continuous function up to the boundary and it is positive. Due to the blow up phenomenon, critical point theory fails here and the standard technique of approximating the problem with a non singular set up comes to our rescue. Another striking feature is the singular exponent being a function here which covers both $0<q(x)<1$ and $q(x)\geq 1$ cases together. In the $0<q(x)<1$ for all $x\in \Om$, we may employ some well known techniques to derive the weak solution. But when $q$ is not bounded by $1$, under an additional assumption of restricting $q$ in $(0,1]$ near the boundary of $\Om$, we were able to show existence of unique weak solution to \eqref{maineqn} i.e. purely singular case. For \eqref{maineqn1}, fixing $0<q(x)<1$ we show that the variational techniques using the Gateaux differentiability of corresponding energy functional can be advantageous to show existence of at least two weak solutions. 
 
 When $q(x)=q$ is a fixed positive constant, singular problems has been investigated in many different contexts in the recent years. In particular, the $p$-Laplace equation
 \begin{equation}\label{pLap}
 -\Delta_p u=h(x,u)\text{ in }\Om,\,u>0\text{ in }\Om,\,u=0\text{ on }\partial\Om,
 \end{equation}
 has been studied concerning the existence, uniqueness and multiplicity results widely for any $1<p<\infty$, including both purely singular and perturbed singular nonlinearity $h$.
 
 In the purely singular case, we refer the reader to Crandall-Rabinowitz-Tartar \cite{CRT}, Boccardo-Orsina \cite{BocOr} for the semilinear case and De Cave \cite{DeCave}, Canino-Sciunzi-Trombetta \cite{Canino} for the quasilinear case respectively. Whereas for the perturbed singular case, one can look into Arcoya-Boccardo \cite{arcoya} for the semilinear case and  Giacomoni-Schindler-Tak\'{a}\v{c} \cite{GST}, Bal-Garain \cite{BG} for the quasilinear case respectively. One may also refer to Garain-Mukherjee \cite{G,GM} and Leggat-Miri \cite{Mirifixed}, dos Santos-Figueiredo-Tavares \cite{Santosetal}, Bal-Garain \cite{BGaniso, Ganiso} for the study of singular problems in the context of weighted and anisotropic $p$-Laplace operator respectively. Furthermore, we motivate readers to read Ghergu-R\u{a}dulescu \cite{GRbook}, Oliva-Orsina-Petitta \cite{Oljama, Olsaim, Oljde, OrPet} for an extensive study of singular elliptic problems.
 
As stated earlier, our main concern in this article is to investigate the existence results in the presence of a variable singular exponent $q$. To the best of our knowledge, the first article to deal with such phenomenon was by Carmona-Aparicio \cite{CMP}, where the authors studied the following singular Laplace equation
 \begin{equation}\label{singLap}
 \begin{split}
 -\Delta u=u^{-q(x)}\text{ in }\Om,\,u>0\text{ in }\Om,\,u=0\text{ on }\partial\Om.
 \end{split}
 \end{equation}
In addition to  \eqref{singLap}, for purely singular nonlinearity, Chu-Gao-Gao \cite{CGG} studied a more general class of second order elliptic equation; Miri \cite{Mirivar} studied anisotropic $p$-Laplace equations; Zhang \cite{Zhang} studied $p(x)$-Laplace equation accomodating variable singular exponent. Alves-Santos-Siqueira \cite{Alves20} obtained existence and uniqueness for a more general $p(x)$-Laplace equation. For the existence results of a system of $p(x)$-Laplace equation, we refer Alves-Moussaoui \cite{Alves18}.  In the purturbed case, Byun-Ko \cite{BKcvpde} discussed multiplicity for a class of $p(x)$-Laplace equation. Papageorgiou-Scapellato \cite{PSzamp} obtained existence results for a purely singular $(p(x),q(x))$-Laplace equation. In the nonlocal setting, problem \eqref{singLap} has been investigated recently by Garain-Mukherjee in \cite{GMnonloc}. For an extensive literature of variable exponent problems, interested readers can go through R\u{a}dulescu-Repov\v{s} \cite{RRvar}.
 
In the same context, we believe that Anisotropic $p$-Laplace equations with singular nonlinearities have been very less understood and considered. We refer to the recent articles by Biset-Mebrate-Mohammed \cite{BMM20}, Farkas-Winkert \cite{PF20} and Farkas-Fiscella-Winkert \cite{PF21} in this regard. We seek some ideas  of approximation from Boccardo-Orsina \cite{BocOr} and Carmona-Aparicio \cite{CMP} for the problem \eqref{maineqn}. On the other side, we obtain multiplicity results for the problem \eqref{maineqn1} using the variational approach from Arcoya-Boccardo \cite{arcoya}.

\textbf{Notation:} We gather below all the notations that has been frequently used in our article-
\begin{itemize}
\item $X:=W_0^{1,p}(\Omega)$.
\item For $u\in X$, we denote by $\|u\|$ to mean $\|u\|_X$ which is defined by \eqref{equinorm}.
\item For a given constant $c$ and a set $S$, by $u\geq c$ in $S$, we mean $u\geq c$ almost everywhere in $S$. Moreover, we write $|S|$ to denote the Lebesgue measure of $S$.
\item $\langle,\rangle$ denotes the standard inner product in $\mathbb{R}^N$.
\item The conjugate exponent of $\theta>1$ is given by $\theta':=\frac{\theta}{\theta-1}$.
\item For $1<p<N$, we denote by $p^*:=\frac{Np}{N-p}$ to mean the critical Sobolev exponent.
\item For $a\in\mathbb{R}$, we denote by $a^+:=\max\{a,0\}$ and $a_-:=\min\{a,0\}$.
\item We write by $c,C$ or $C_i$ for $i\in\mathbb{N}$ to mean a positive constant which may vary from line to line or even in the same line. If a constant $C$ depends on $r_1,r_2,\ldots$, we denote it by $C(r_1,r_2,\ldots)$.
\end{itemize}
Now, we define the notion of weak solutions as follows.\\
\textbf{Weak solution:} We say that $u\in X$ is a weak solution of the problem \eqref{maineqn}, if $u>0$ in $\Omega$ and for every $\omega\Subset\Om$, there exists a positive constant $c_{\omega}$ such that $u\geq c_{\omega}>0$ in $\omega$, satisfying
\begin{equation}\label{psingtest}
\int_{\Om}H(\nabla u)^{p-1}\nabla_{\eta}H(\nabla u)\nabla\phi\,dx=\int_{\Om}\frac{f(x)}{u^{q(x)}}\phi\,dx,\quad\forall\,\phi\in C_c^{1}(\Omega).
\end{equation}
Analogously, we say that $u\in X$ is a weak solution of the problem \eqref{maineqn1}, if $u>0$ in $\Omega$ and for every $\omega\Subset\Om$, there exists a positive constant $c_{\omega}$ such that $u\geq c_{\omega}>0$ in $\omega$, satisfying
\begin{equation}\label{pursingtest}
\int_{\Om}H(\nabla u)^{p-1}\nabla_{\eta}H(\nabla u)\nabla\phi\,dx=\int_{\Om}\Big(\frac{\lambda}{u^{q(x)}}+u^r\Big)\phi\,dx,\quad\forall\,\phi\in C_c^{1}(\Omega).
\end{equation}

Our main results in this article reads includes the following.
\begin{Theorem}\label{thm1}
Let $2\leq p<\infty$ and $q\in C(\overline{\Omega})$ be positive.
\begin{enumerate}
\item[(a)] \textbf{Uniqueness:} Then for any nonnegative $f\in L^1(\Omega)\setminus\{0\}$, the problem \eqref{maineqn} admits at most one weak solution in $X$.
\item[(b)] \textbf{Existence:} Assume that there exists a constant $\delta>0$ such that $0<q(x)\leq 1$ in $\Om_{\delta}$, where
$$
\Om_{\delta}:=\{x\in\Om:\text{dist}\,(x,\partial\Om)<\delta\}.
$$
Then the problem $(\ref{maineqn})$ admits a unique weak solution $u\in X$, for any nonnegative $f\in L^m(\Om)\setminus\{0\}$, where
\[
m=
\begin{cases}
(p^*)',\text{ if }2\leq p<N,\\
>1,\text{ if }p=N,\\
1,\text{ if }p>N.
\end{cases}
\]
\end{enumerate}
\end{Theorem}
In particular, we have the following existence and uniqueness result for any $1<p<\infty$.
\begin{Theorem}\label{thm1new}
If
$$
\Delta_{H,p}\equiv\Delta_{p}\text{ or }\mathcal{S}_p,
$$
as given by \eqref{ex}, then Theorem \ref{thm1} holds for any $1<p<\infty$.  
\end{Theorem}

\begin{Theorem}\label{thm3}
Let $2\leq p<N$ and $q\in C(\overline{\Om})$ be such that $0<q(x)<1$ for all $x\in\overline{\Om}$. Then there exists $\Lambda>0$ such that for all $\la\in(0,\Lambda)$, the problem $\eqref{maineqn1}$ admits at least two distinct weak solutions in $X$, for any $r\in(p-1,p^*-1)$.
\end{Theorem}

In particular, we have the following multiplicity result for any $1<p<N$.
\begin{Theorem}\label{thm3new}
If
$$
\Delta_{H,p}\equiv\Delta_{p}\text{ or }\mathcal{S}_p,
$$
as given by \eqref{ex}, then Theorem \ref{thm3} holds for any $1<p<N$.    
\end{Theorem}

We have segmented our paper as follows: In Section $2$, we discuss some auxiliary results in our functional setting. In Section $3$ and Section $4$, we obtain some preliminary results for the proof of Theorem \ref{thm1}, Theorem \ref{thm1new} and Theorem \ref{thm3}, Theorem \ref{thm3new} respectively. Finally, the proof of all main results has been collected in Section $5$.

\section{Auxiliary results}
This section is devoted to preliminary ideas which will be helpful to advance towards establishing our principle results.
We fix our assumption to $1<p<\infty$, unless otherwise mentioned. We state some results below which are taken from Farkas-Winkert \cite[Proposition 2.1]{PF20} and Xia \cite[Proposition 1.2]{Xiathesis}.
\begin{Lemma}\label{Happ}
For every $x\in\mathbb{R}^N\setminus\{0\}$ and $t\in\mathbb{R}\setminus\{0\}$ we have
\begin{enumerate}
\item[(A)] $x\cdot\nabla_{\eta} H(x)=H(x)$. 
\item[(B)] $\nabla_{\eta}H(tx)=\text{sign}(t)\nabla_{\eta}H(x)$.
\item[(C)] $|\nabla_{\eta} H(x)|\leq C$, for some positive constant $C$.
\item[(D)] $H$ is stricly convex.
\end{enumerate}
\end{Lemma}

\begin{Corollary}\label{regrmk}
On account of  Lemma \ref{Happ} and Remark \ref{hypormk2}, we have the following relations-
\begin{equation}\label{lbd}
H(x)^{p-1}\nabla_{\eta}H(x)\cdot x=H(x)^p\geq C_1|x|^p,\quad\forall\,x\in\mathbb{R}^N,
\end{equation}
\begin{equation}\label{ubd}
\big|H(x)^{p-1}\nabla_{\eta}H(x)\big|\leq C_2|x|^{p-1},\quad\forall\,x\in\mathbb{R}^N\text{ and}
\end{equation}
\begin{equation}\label{homo}
H(tx)^{p-1}\nabla_{\eta}H(tx)=|t|^{p-2}t H(x)^{p-1}\nabla_{\eta}H(x),\quad\forall\,x\in\mathbb{R}^N\text{ and }t\in\mathbb{R}\setminus\{0\}.
\end{equation}
\end{Corollary}

The following result has been adapted from Ohta \cite[Proposition 4.6]{Ohta}.
\begin{Lemma}\label{Ohta}
For every $x,y\in\mathbb{R}^N$, there exists a constant $C\geq 1$ such that
\begin{equation}\label{Ohtalem}
H\left(\frac{x+y}{2}\right)^2+\frac{1}{4C^2}H(x-y)^2\leq \frac{H(x)^2+H(y)^2}{2}.
\end{equation}
\end{Lemma}

Lemma \ref{algpre}-\ref{alg} given below follows from arguments given in the proof of Xia \cite[Lemma $3.1-3.2$]{Xia}.
\begin{Lemma}\label{algpre}
Let $2\leq p<\infty$. Then, for every $x,y\in\mathbb{R}^N$,
\begin{equation}\label{algpre1}
H(x)^p\geq H(y)^p+pH(y)^{p-1}\nabla_{\eta}H(y)(x-y).
\end{equation}
Moreover, there exists a positive constant $c=c(C,p)$, where $C$ is given by Lemma \ref{Ohta} such that for every $x,y\in\mathbb{R}^N$,
\begin{equation}\label{algpre2}
\begin{split}
H(x)^p&\geq
H(y)^p+pH(y)^{p-1}\nabla_{\eta}H(y)(x-y)+cH(x-y)^p.
\end{split}
\end{equation}
\end{Lemma}

\begin{proof}
Firstly, we observe that $H^p$ is convex. Indeed, for any $x,y\in\mathbb{R}^N$ and $s\in[0,1]$, by the convexity of $H$ (follows from Lemma \ref{Happ}), we have
\begin{equation*}\label{con}
\begin{split}
H(sx+(1-s)y)^p&\leq\left(sH(x)+(1-s)H(y)\right)^p\\
&\leq sH(x)^p+(1-s)H(y)^p,
\end{split}
\end{equation*}
where in the final step above, we have used the convexity of $|\cdot|^p$. Hence the estimate \eqref{algpre1} follows. Next we prove \eqref{algpre2}. For any $a,b\geq 0$ and $p\geq 2$, we have the following elementary inequality (see Xia \cite{Xia})
\begin{equation}\label{useineq}
(a^p+b^p)\leq(a^2+b^2)^\frac{p}{2}.
\end{equation}
Let $C\geq 1$ be given by Lemma \ref{Ohta}. Choosing $a=H\big(\frac{x+y}{2}\big)$ and $b=H\big(\frac{x-y}{2C}\big)$ in \eqref{useineq} and using Lemma \ref{Ohta} we get
\begin{equation}\label{pgeq2}
\begin{split}
H\left(\frac{x+y}{2}\right)^p+H\left(\frac{x-y}{2C}\right)^p&\leq\left(\frac{H(x+y)^2}{4}+\frac{1}{4C^2}H(x-y)^2)\right)^\frac{p}{2}\\
&\leq\left(\frac{H(x)^2+H(y)^2}{2}\right)^\frac{p}{2}\leq\frac{H(x)^p}{2}+\frac{H(y)^p}{2},
\end{split}
\end{equation}
where in the last line above we have used the convexity of $|\cdot|^\frac{p}{2}$ for $p\geq 2$.
Moreover, by \eqref{algpre1}, we have
\begin{equation}\label{pgeq2app}
H\left(\frac{x+y}{2}\right)^p\geq H(y)^p+\frac{p}{2}H(y)^{p-1}\nabla_{\eta}H(y)(x-y).
\end{equation}
Hence using \eqref{pgeq2app} in \eqref{pgeq2} the estimate \eqref{algpre2} follows.
\end{proof}

Our next Lemma is a well know identity for the $p$-Laplacian operators, we prove it for Finsler $p$-Laplacian (a general case).
\begin{Lemma}\label{alg}
Let $2\leq p<\infty$. Then, for every $x,y\in\mathbb{R}^N$, there exists a positive constant $c=c(C,p)$, where $C$ is given by Lemma \ref{Ohta} such that
\begin{equation}\label{algineq}
\begin{split}
\langle H(x)^{p-1}\nabla_{\eta}H(x)-H(y)^{p-1}\nabla_{\eta}H(y),x-y\rangle&\geq
cH(x-y)^p.
\end{split}
\end{equation}
\end{Lemma}

\begin{proof}
From \eqref{algpre2} in Lemma \ref{algpre}, for every $x,y\in\mathbb{R}^N$ we have
\begin{equation}\label{alg1}
\begin{split}
H(x)^p&\geq H(y)^p+pH(y)^{p-1}\nabla_{\eta}H(y)(x-y)+cH(x-y)^p,
\end{split}
\end{equation}
and
\begin{equation}\label{alg2}
\begin{split}
H(y)^p&\geq H(x)^p+pH(x)^{p-1}\nabla_{\eta}H(x)(y-x)+cH(y-x)^p.
\end{split}
\end{equation}
Adding \eqref{alg1}, \eqref{alg2} and using $(H2)$ the estimate \eqref{algineq} follows.
\end{proof}

\begin{Remark}\label{algrmk}
When $2\leq p<\infty$ and $H(x)=H_2(x)=|x|$ as given by \eqref{ex1}, Lemma \ref{alg} coincides with the well-known algebraic inequality from Lemma \ref{AI}.
\end{Remark}

\begin{Corollary}\label{regapp}
As a consequence of Lemma \ref{alg}, we have
\begin{equation}\label{regapp1}
\langle H(x)^{p-1}\nabla_{\eta}H(x)-H(y)^{p-1}\nabla_{\eta}H(y),x-y\rangle>0,\quad\forall\,x\neq y\in\mathbb{R}^N.
\end{equation}
\end{Corollary}

The following inequalities has been borrowed from\cite[Proposition $1.3$ and $(1.2)$]{Xiathesis}.
\begin{Lemma}\label{dualthm}
\begin{enumerate}
\item[(D1)] Cauchy-Schwartz inequality: \begin{equation}\label{CSineq}
\langle x,\xi\rangle\leq H(x)H_0(\xi),\quad\forall\,x,\xi\in\mathbb{R}^N.
\end{equation}
\item[(D2)] For all $x\in\mathbb{R}^N\setminus\{0\}$, it holds 
\begin{equation}\label{dual1}
H_0\big(\nabla H(x)\big)=1.
\end{equation}
\end{enumerate}
\end{Lemma}
 Below is an important output of above inequalities.
\begin{Corollary}\label{dualthmcor}
For any $x,y \in \mb R^N$, it holds that
\begin{equation}\label{dualthmapp3}
\langle H(x)^{p-1}\nabla_{\eta}H(x)-H(y)^{p-1}\nabla_{\eta}H(y),x-y\rangle\geq \left(H(x)^{p-1}-H(y)^{p-1}\big)\big(H(x)-H(y)\right).
\end{equation}
\end{Corollary}
\begin{proof}
Using Lemma \ref{dualthm} we observe that
\begin{equation}\label{dualthmapp1}
H(x)^{p-1}\langle\nabla_{\eta}H(x),y\rangle\leq H(x)^{p-1} H_0(\nabla_{\eta}H(x))H(y)=H(x)^{p-1}H(y)\text{ and}
\end{equation}
\begin{equation}\label{dualthmapp2}
H(y)^{p-1}\langle\nabla_{\eta}H(y),x\rangle\leq H(y)^{p-1} H_0(\nabla_{\eta}H(y))H(x)=H(y)^{p-1}H(x),
\end{equation}
holds for every $x,y\in\mathbb{R}^N$. Hence using \eqref{dualthmapp1} and \eqref{dualthmapp2}, for all $x,y\in\mathbb{R}^N$ we obtain
\begin{equation*}
\begin{split}
&\langle H(x)^{p-1}\nabla_{\eta}H(x)-H(y)^{p-1}\nabla_{\eta}H(y),x-y\rangle\\
&=H(x)^p+H(y)^p-H(x)^{p-1}\langle\nabla_{\eta}H(x),y\rangle-H(y)^{p-1}\langle\nabla_{\eta}H(y),x\rangle\\
&\geq H(x)^{p-1}+H(y)^{p-1}-H(x)^{p-1}H(y)-H(y)^{p-1}H(x)\\
&=\big(H(x)^{p-1}-H(y)^{p-1}\big)\big(H(x)-H(y)\big).
\end{split}
\end{equation*}
\end{proof}

The following result follows from Belloni-Ferone-Kawohl \cite[Theorem 3.1]{BFKzamp}.
\begin{Lemma}\label{eigenfn}
There exists a positive eigenfunction $e_1\in X\cap L^{\infty}(\Om)$ such that $\|e_1\|_{L^\infty(\Om)}=1$ corresponding to the first eigenvalue $\lambda_1>0$ satisfying the equation
\begin{equation}\label{eigeneqn}
-\Delta_{H,p}v=\lambda_{1}|v|^{p-2}v\text{ in }\Om,\;\;
v=0\text{ in }\partial\Om.
\end{equation}
\end{Lemma}

For the following Sobolev embedding, see Evans \cite{Evans}.
\begin{Lemma}\label{embedding}
The inclusion map
\[
W_0^{1,p}(\Omega)\hookrightarrow
\begin{cases}
L^t(\Om),\text{ for }t\in[1,p^{*}],\text{ if }1<p<N,\\
L^t(\Om),\text{ for }t\in[1,\infty),\text{ if }p=N,\\
C(\overline{\Om}),\text{ if }p>N,
\end{cases}
\]
is continuous. Moreover, the above mapping is compact except for $r=p^*$, when $1<p<N$.
\end{Lemma}

Next, we state the algebraic inequality from Peral \cite[Lemma A.0.5]{Pe}.

\begin{Lemma}\label{AI}
For any $a,b\in\mathbb{R}^N$, there exists a constant $C=C(p)>0$, such that
\begin{equation}\label{ALGin}
\langle |a|^{p-2}a-|b|^{p-2}b, a-b \rangle\geq
\begin{cases}
C|a-b|^p,\text{ if }2\leq p<\infty,\\
C\frac{|a-b|^2}{(|a|+|b|)^{2-p}},\text{ if }1<p<2.
\end{cases}
\end{equation}
\end{Lemma}

Our next result ensures that test functions in \eqref{psingtest} and \eqref{pursingtest} can be chosen from the space $X$ itself.
\begin{Lemma}\label{testfn}
Let $u\in X$ be a weak solution of the problem \eqref{maineqn} or \eqref{maineqn1}, then  \eqref{psingtest} or \eqref{pursingtest} holds, for every $\phi\in X$ respectively.
\end{Lemma}
\begin{proof}
Let $u\in X$ solves the problem $\eqref{maineqn}$ or \eqref{maineqn1}. Suppose 
$$
g(x,u)=f(x)u^{-q(x)}\text{ or }\lambda u^{-q(x)}+u^r.
$$
Therefore, for every $\phi\in C_c^1(\Omega)$, from \eqref{psingtest} and \eqref{pursingtest} we have
\begin{equation}\label{smthtest}
\int_{\Om}H(\nabla u)^{p-1}\nabla_{\eta}H(\nabla u)\nabla\phi\,dx=\int_{\Omega}g(x,u)\phi\,dx.
\end{equation}
Now for every $\psi\in X$, there exists a sequence of function $0\leq \psi_n\in C_c^{1}(\Omega)\to |\psi|$ strongly in $X$ as $n\to\infty$ and pointwise almost everywhere in $\Omega$. We observe that
\begin{equation}\label{unique}
\begin{split}
\Big|\int_{\Omega}g(x,u)\psi\,dx\Big|&\leq \int_{\Omega}g(x,u)|\psi|\,dx\leq\liminf_{n\to\infty}\int_{\Omega}g(x,u)\psi_n\,dx\\
&=\liminf_{n\to\infty}<-\Delta_{H,p} u,\psi_n>\leq C\|u\|^{p-1}\lim_{n\to\infty}\|\psi_n\|\\
&\leq C\|u\|^{p-1}\||\psi|\|\leq C\|u\|^{p-1}\|\psi\|,
\end{split}
\end{equation}
for some positive constant $C$. Let $\phi\in X$, then there exists a sequence $\{\phi_n\}\subset C_c^{1}(\Omega)$ which converges to $\phi$ strongly in $X$. We claim that
$$
\lim_{n\to\infty}\int_{\Omega}g(x,u)\phi_n\,dx=\int_{\Omega}g(x,u)\phi\,dx.
$$
Indeed, using $\psi=\phi_n-\phi$ in \eqref{unique}, we obtain
\begin{equation}\label{lhslim}
\lim_{n\to\infty}\Big|\int_{\Omega}g(x,u)(\phi_n-\phi)\,dx\Big|\leq C||u||^{p-1}\lim_{n\to\infty}||\phi_n-\phi||=0.
\end{equation}
Again, since $\phi_n\to\phi$ strongly in $X$ as $n\to\infty$, we have
\begin{equation}\label{rhslim}
\lim_{n\to\infty}H(\nabla u)^{p-1}\nabla_{\eta}H(\nabla u)\nabla(\phi_n-\phi)\,dx=0.
\end{equation}
Hence, using \eqref{lhslim} and \eqref{rhslim} in \eqref{smthtest} the result follows.
\end{proof}

\section{Preliminaries for the proof of Theorem \ref{thm1} and Theorem \ref{thm1new}}
We present proof of our first main result in this section.
Here again, we assume $1<p<\infty$,
unless otherwise mentioned. For $n\in\mathbb{N}$ and nonnegative $f\in L^1(\Om)\setminus\{0\}$, let us define $f_n(x):=\min\{f(x),n\}$ and consider the following approximated problem
\begin{equation}\label{mainapprox}
\left.
\begin{aligned}
-\Delta_{H,p}u&=\frac{f_n(x)}{(u^{+}+\frac{1}{n})^{q(x)}}\text{ in }\Omega,\\
u&=0\text{ on }\partial\Omega.
\end{aligned}\right\}
\end{equation}
First, we prove the following useful result.
\begin{Lemma}\label{auxresult}
Let $g\in L^{\infty}(\Om)\setminus\{0\}$ be a nonnegative function in $\Om$. Then there exists a unique solution $u\in X\cap L^{\infty}(\Om)$ of the problem
\begin{equation}\label{auxresulteqn}
\begin{split}
-\Delta_{H,p}u=g\text{ in }\Om,\,u>0\text{ in }\Om,\,u=0\text{ on }\partial\Om,
\end{split}
\end{equation}
such that for every $\omega\Subset\Om$, there exists a constant $c_{\omega}$ satisfying $u\geq c_{\omega}>0$ in $\omega$.
\end{Lemma}
\begin{proof}
\textbf{Existence:} We define the energy functional $J: X\to\mathbb{R}$ by
$$
J(u):=\frac{1}{p}\int_{\Om}H(\nabla u)^p\,dx-\int_{\Om}gu\,dx.
$$
Noting Lemma \ref{Happ} and Lemma \ref{embedding}, it can be easily seen that $J$ is coercive, weakly lower semicontinuous and a strictly convex functional. Therefore, $J$ has a unique minimizer $u$ in $X$. Since $J\in C^1(X)$, so $u\in X$ is the unique solution of the equation
\begin{equation}\label{auxeqn}
-\Delta_{H,p}u=g\text{ in }\Om,\,u=0\text{ on }\partial\Om.
\end{equation}
Noting $g\geq 0$ and choosing $u_-:=\min\{u,0\}$ as a test function in \eqref{auxeqn}, by {\eqref{lbd} we obtain
\begin{equation*}
\begin{split}
C_1\int_{\Omega}|\nabla u_-|^p\,dx&\leq\int_{\Omega}H\big(\nabla u_-\big)^p\,dx=\int_{\Omega}H(\nabla u)^{p-1}\nabla_{\eta}H(\nabla u)\nabla u_-\,dx=\int_{\Omega}gu_-\,dx\leq 0,
\end{split}
\end{equation*}
which gives, $u\geq 0$ in $\Omega$.\\
\textbf{Boundedness:} For any $k\geq 1$, we define a subset of $\Omega$ given by
$$
A(k):=\{x\in\Omega:u(x)\geq k\text{ in }\Omega\}.
$$
Choosing $\phi_k:=(u-k)^+=\max\{u-k,0\}$ as a test function in \eqref{auxeqn} and using the continuity of the mapping $X\hookrightarrow L^l(\Omega)$ for some $l>p$ from Lemma \ref{embedding} along with \eqref{lbd} we obtain
\begin{equation*}
\begin{split}
\int_{\Omega}H(\nabla\phi_k)^p\,dx=\int_{\Om}g\phi_k~dx\leq \|g\|_{L^{\infty}(\Om)}\int_{A(k)}(u-k)\,dx\leq C_0\|g\|_{L^{\infty}(\Om)}|A(k)|^\frac{l-1}{l}\|\phi_k\|,
\end{split}
\end{equation*}
where $C_0$ is the Sobolev constant. Hence, we have
\begin{equation}\label{new}
    \|\phi_k\|^{p-1} \leq C|A(k)|^\frac{l-1}{l},
\end{equation}
for some positive constant $C$ depending on $\|g\|_{L^{\infty}(\Om)}$. We choose $h$ such that $1<k<h$ then using \eqref{new}, $(u(x)-k)^p \geq (h-k)^p$ on $A(h)$ and $A(h)\subset A(k)$ we get
\begin{align*}
(h-k)^p|A(h)|^\frac{p}{l}&\leq\left(\int_{A(h)}(u(x)-k)^l\,dx\right)^\frac{p}{l}\leq\left(\int_{A(k)}(u(x)-k)^l~dx\right)^\frac{p}{l}\\
&\leq\|\phi_k\|^p\leq C|A(k)|^\frac{p(l-1)}{l(p-1)}.
\end{align*}
Hence, we obtain
$$
|A(h)|\leq\frac{C}{(h-k)^l}|A(k)|^\frac{l-1}{p-1}.
$$
Therefore, observing that $\frac{l-1}{p-1}>1$, using Kinderlehrer-Stampacchia \cite[Lemma B.1]{Stam} we have
$$
\|u\|_{L^\infty(\Omega)}\leq C,
$$
for some positive constant $C$ depending on $\|g\|_{L^{\infty}(\Om)}$.}\\
\textbf{Positivity:} Moreover since $g\neq 0$, we have $u\neq 0$ in $\Omega$. Noting Corollary \ref{regrmk} and Corollary \ref{regapp}, we can apply Heinonen-Kilpel\"{a}ine-Martio \cite[Theorem 3.59]{Juh} so that for every $\omega\Subset\Omega$, there exists a constant $c_{\omega}>0$ such that $u\geq c_{\omega}>0$ in $\Omega$. Thus $u>0$ in $\Omega$.
\end{proof}

Then, we have the following result concerning the problem \eqref{mainapprox}.
\begin{Lemma}\label{approxexi}
Let $2\leq p<\infty$. Then for every $n\in\mathbb{N}$, there exists a unique positive solution $u_n\in W_{0}^{1,p}(\Omega)\cap L^{\infty}(\Omega)$ to the problem \eqref{mainapprox} such that $u_{n+1}\geq u_n$ in $\Om$, for every $n\in\mathbb{N}$. Moreover for every $\omega\Subset\Om$, there exists a constant $c_{\omega}>0$ (independent of $n$) such that $u_n \geq c_{\omega}>0$ in $\omega$.
\end{Lemma}
\begin{proof}
\textbf{Step $1$. (Existence)} We observe that for every $n\in\mathbb{N}$ and $v\in L^p(\Omega)$, the function 
$$
g_n(x):=\frac{f_n(x)}{\big(v^{+}+\frac{1}{n}\big)^{q(x)}}\in L^\infty(\Om)\setminus\{0\}
$$
is nonnegative. Therefore, by Lemma \ref{auxresult} we can define the map $T:L^p(\Omega)\to L^p(\Omega)$ by
$$
T(v)=w_n,
$$ where $w_n\in X\cap L^\infty(\Om)$ is the unique solution to the problem
\begin{equation}\label{approxexiaux}
\left.
\begin{aligned}
-\Delta_{H,p}w_n &=\frac{f_n(x)}{(v^{+}+\frac{1}{n})^{q(x)}}\text{ in }\Omega,\\
w_n&>0\text{ in }\Om,\,w_n=\text{ on }\partial\Om,
\end{aligned}\right\}
\end{equation}
such that for every $\omega\Subset\Om$ there exists a positive constant $c_{\omega}$ satisfying $w_n\geq c_{\omega}>0$ in $\omega$. Now, using $w_n$ as a test function in the problem \eqref{approxexiaux} and the fact that $q\in C(\overline{\Omega})$, we obtain
\begin{equation}\label{bdd}
\|w_n\|\leq Cn^\frac{\|q\|_{L^\infty(\Om)}+1}{p-1},
\end{equation}
for some constant $C=C(p,N,\Om)$. We define
$$
S:=\{v\in L^p(\Omega):\lambda T(v)=v\text{ for }0\leq\lambda\leq 1\}.
$$
Indeed, let $v_1,v_2\in S$. Then using \eqref{bdd} and Lemma \ref{embedding}
\begin{equation*}
\begin{split}
\|v_1-v_2\|_{L^p(\Omega)}&=\lambda\|T(v_1)-T(v_2)\|_{L^p(\Omega)}\\
&\leq 2\lambda CC_1 n^\frac{\|q\|_{L^\infty(\Om)}+1}{p-1},
\end{split}
\end{equation*}
where $C=C(p,N,\Om)$ is given by \eqref{bdd} and $C_1$ is the Sobolev constant. Therefore, the set $S$ is bounded in $L^p(\Omega)$.
To prove the continuity of $T$, let $v_k\to v$ strongly in $L^p(\Omega)$ and $T(v_k)=w_{n,k}$. Then, we have
\begin{equation}\label{cont1}
\int_{\Omega}H(\nabla w_{n,k})^{p-1} \nabla_{\eta}H(\nabla w_{n,k})\nabla(w_{n,k}-u_n)\,dx=\int_{\Omega}\frac{f_n(x)}{(v_k^{+}+\frac{1}{n})^{q(x)}}(w_{n,k}-w_n)\,dx,
\end{equation}
and
\begin{equation}\label{cont2}
\int_{\Omega}H(\nabla w_{n,k})^{p-1}\nabla_{\eta}H(\nabla w_{n,k})\nabla(w_{n,k}-w_n)\,dx=\int_{\Omega}\frac{f_n(x)}{(v^{+}+\frac{1}{n})^{q(x)}}(w_{n,k}-w_n)\,dx.
\end{equation}
Subtracting \eqref{cont2} from \eqref{cont1}, since $2\leq p<\infty$, we can use Lemma \ref{alg} to obtain
\begin{equation}\label{cont4}
\begin{split}
&\int_{\Omega}f_n(x)\left\{\left(v_k^{+}+\frac{1}{n}\right)^{-q(x)}-\left(v^{+}+\frac{1}{n}\right)^{-q(x)}\right\}(w_{n,k}-w_n)\,dx\\
&=\int_{\Omega}\left\{H(\nabla w_{n,k})^{p-1}\nabla_{\eta}H(\nabla w_{n,k})-H(\nabla w_n)^{p-1}\nabla_{\eta}H(\nabla w_n)\right\}\nabla(w_{n,k}-w_n)\,dx\\
&\geq c\int_{\Omega}H\left(\nabla(w_{n,k}-w_n)\right)^p\,dx,
\end{split}
\end{equation}
for some positive constant $c$. From H\"older's inequality, we obtain
\begin{equation}\label{cont6}
\begin{split}
&\int_{\Om}f_n(x)\left\{\big(v_k^{+}+\frac{1}{n}\big)^{-q(x)}-\left(v^{+}+\frac{1}{n}\right)^{-q(x)}\right\}(w_{n,k}-w_n)\,dx\\
&\leq n\left(\int_{\Om}\left|{\left(v_k^{+}+\frac{1}{n}\right)^{-q(x)}}-\left(v^{+}+\frac{1}{n}\right)^{-q(x)}\right|^\frac{p}{p-1}\,dx\right)^\frac{p-1}{p}\|w_{n,k}-w_n\|_{L^p(\Om)}.
\end{split}
\end{equation}
Since $v_k\to v$ strongly in $L^p(\Omega)$, it follows that upto a subsequence $v_k\to v$ almost everywhere in $\Omega$. Also, it is easy to observe that
$$
\Big|\big(v_k^{+}+\frac{1}{n}\big)^{-q(x)}-\big(v^{+}+\frac{1}{n}\big)^{-q(x)}\Big|\leq 2 n^{||q||_{L^\infty(\Om)}}.
$$
Since the limit is independent of the choice of the subsequence, applying the Lebsegue dominated convergence theorem, we obtain
\begin{equation}\label{cont7}
\lim_{k\to\infty}\int_{\Omega}\left|\left(v_k^{+}+\frac{1}{n}\right)^{-q(x)}-\left(v^{+}+\frac{1}{n}\right)^{-q(x)}\right|^\frac{p}{p-1}\,dx=0.
\end{equation}
Thus using \eqref{cont4} along with \eqref{cont6} and \eqref{cont7} it follows that $w_{n,k}\to w_n$ strongly in $X$. As a consequence, the mapping $T$ is continuous. The compactness of $T$ follows from the fact \eqref{bdd} and the compactness of the mapping $X\hookrightarrow L^p(\Om)$ which follows from Lemma \ref{embedding}. Thus by the Schauder fixed point theorem, $T$ has a fixed point, say $u_n\in X$ which  solves the problem \eqref{mainapprox}. Note that $u_n\geq 0$ in $\Om$.\\
\textbf{Step $2$. (Monotonicity)} Choosing $(u_n-u_{n+1})^{+}:=\max\{u_n-u_{n+1},0\}$ as a test function in \eqref{mainapprox} we have
\begin{equation}\label{mono1}
\int_{\Omega}H(\nabla u_n)^{p-1}\nabla_{\eta}H(\nabla u_n)\nabla(u_n-u_{n+1})^{+}\,dx=\int_{\Omega}\frac{f_n(x)}{\left(u_n+\frac{1}{n}\right)^{q(x)}}(u_n-u_{n+1})^{+}\,dx,
\end{equation}
and
\begin{equation}\label{mono2}
\int_{\Omega}H(\nabla u_{n+1})^{p-1}\nabla_{\eta}H(\nabla u_{n+1})\nabla(u_n-u_{n+1})^{+}\,dx=\int_{\Omega}\frac{f_{n+1}(x)}{\left(u_{n+1}+\frac{1}{n+1}\right)^{q(x)}}(u_n-u_{n+1})^{+}\,dx.
\end{equation}
Subtracting \eqref{mono2} from \eqref{mono1}, since $2\leq p<\infty$, using Lemma \ref{alg} we obtain
\begin{equation}\label{monosub}
\begin{split}
&\int_{\Omega}\left\{\frac{f_n(x)}{\left(u_n+\frac{1}{n}\right)^{q(x)}}-\frac{f_{n+1}(x)}{\left(u_{n+1}+\frac{1}{n+1}\right)^{q(x)}}\right\}(u_n-u_{n+1})^{+}\,dx\\
&=\int_{\Omega}\left\{H(\nabla u_n)^{p-1}\nabla_{\eta}H(\nabla u_n)-H(\nabla u_{n+1})^{p-1}\nabla_{\eta}H(\nabla u_{n+1})\right\}\nabla(u_n-u_{n+1})^{+}\,dx\\
&\geq c\int_{\Omega}H\left(\nabla(u_n-u_{n+1})^{+}\right)^p\,dx,
\end{split}
\end{equation}
 for some positive constant $c$. Since $f_n(x)\leq f_{n+1}(x)$ for almost every $x\in\Omega$, we get
\begin{equation}\label{monoeqn}
\begin{split}
&\int_{\Omega}\left\{\frac{f_n(x)}{\big(u_n+\frac{1}{n}\big)^{q(x)}}-\frac{f_{n+1}(x)}{\big(u_{n+1}+\frac{1}{n+1}\big)^{q(x)}}\right\}(u_n-u_{n+1})^{+}\,dx\\
&\leq\int_{\Omega}f_{n+1}(x)\frac{\big(u_{n+1}+\frac{1}{n+1}\big)^{q(x)}-\big(u_n+\frac{1}{n}\big)^{q(x)}}{\big(u_n+\frac{1}{n}\big)^{q(x)}\big(u_{n+1}+\frac{1}{n+1}\big)^{q(x)}} (u_n-u_{n+1})^{+}\,dx \leq 0.
\end{split}
\end{equation}
Hence from \eqref{monosub} and \eqref{monoeqn}, we have $u_{n+1}\geq u_n$ in $\Om$.\\
\textbf{Step $3$. (Uniqueness)} Let $n\in\mathbb{N}$ and $u_n,v_n\in X\cap L^\infty(\Om)$ solves the problem \eqref{mainapprox}. Then choosing $(u_n-v_n)^+:=\max\{u_n-v_n,0\}$ as a test function in \eqref{mainapprox} and proceeding similarly as the proof of Step $2$ above, we get $u_n\geq v_n$ in $\Omega$. Choosing $(v_n-u_n)^+:=\max\{v_n-u_n,0\}$ as a test function in \eqref{mainapprox}, we get $v_n\geq u_n$ in $\Omega$. This gives that $u_n=v_n$ in $\Om$.\\
\textbf{Step $4$. (Positivity)} From Step $1$, we infer that there exists a constant $c_{\omega}>0$ such that $u_1\geq c_{\omega}>0$, for every $\omega\Subset\Om$. Hence taking into account Step $2$, we have $u_n\geq c_{\omega}>0$ in $\omega$ for every $\omega\Subset\Om$, where $c_\omega>0$ is independent of $n$. Thus $u_n>0$ in $\Omega$ for all $n\in\mathbb{N}$, from which the result follows.
\end{proof}
\begin{Remark}\label{sinrmk}
We emphasize that if $\Delta_{H,p}=\Delta_p$ or $\mathcal{S}_p$ as given by \eqref{ex}, then noting Lemma \ref{AI}, the results in Lemma \ref{approxexi} will hold for any $1<p<\infty$.
\end{Remark}

\section{Preliminaries for the proof of Theorem \ref{thm3} and Theorem \ref{thm3new}}
Throughout this section, we assume $1<p<N$, unless otherwise mentioned. Here we prove some preliminary results required to prove Theorem \ref{thm3}-\ref{thm3new}.

Let us denote the energy functional $I_\la: X \to \mb R\cup \{\pm \infty\}$ corresponding to the problem \eqref{maineqn1} by
\[I_\la(u) := \frac{1}{p}\int_{\Omega}H(\nabla u)^p\,dx -\la \int_\Om \frac{(u^+)^{1-q(x)}(x)}{1-q(x)}~dx -\frac{1}{r+1}\int_\Om (u^+)^{r+1}~dx.\]
Now for $\epsilon>0$, we consider the following approximated problem
\begin{equation*}
(P_{\la,\e})\left\{
\begin{aligned}
  -\Delta_{H,p} u &=  \frac{\la}{(u^+ +\e)^{q(x)}}+ (u^+)^r\;\text{in}\; \Om,\\
 u&=0 \; \text{ on }\; \partial\Om,
\end{aligned}\right.
\end{equation*}
for which the corresponding energy functional is given by
\[I_{\la,\e}(u) = \frac{1}{p}\int_{\Omega}H(\nabla u)^p\,dx -\la\int_\Om \frac{[(u^+ +\e)^{1-q(x)}-\e^{1-q(x)}]}{1-q(x)}~dx -\frac{1}{r+1}\int_\Om (u^+)^{r+1}~dx.\]
It is easy to verify that $I_{\la,\e}\in C^1(X,\mb R)$, $I_{\la,\epsilon}(0)=0$ and $I_{\la,\epsilon}(v)\leq I_{0,\epsilon}(v)$, for all $ v \in X$.

Our next Lemma states that $I_{\la,\e}$ satisfies the Mountain Pass Geometry.
\begin{Lemma}\label{MP-geo}
There exists $R>0,\,\rho>0$ and $\Lambda>0$ depending on $R$ such that
$$
\inf\limits_{\|v\|\leq R}I_{\la,\e}(v)<0\;\text{and}\;
\inf\limits_{\|v\|=R}I_{\la,\e}(v)\geq \rho,\text{ for }\la\in(0,\Lambda).
$$
Moreover, there exists $T>R$ such that
$
I_{\la,\e}(Te_{1})<-1$ for $\la\in (0,\La)$, where $e_1$ is given by Lemma \ref{eigenfn}.
\end{Lemma}
\begin{proof}
We fix $l=|\Om|^{\frac{1}{\left(\frac{p^*}{r+1}\right)'}}$. Then using H\"{o}lder's inequality and Lemma \ref{embedding}, for any $v\in X$ we get
\begin{equation*}\label{MP1}
\int_\Om (v^+)^{r+1}~dx \leq \left( \int_\Om |v|^{p^*}\right)^{\frac{r+1}{p*}} |\Om|^{\frac{1}{(\frac{p^*}{r+1})'}}\leq Cl\|v\|^{r+1},
\end{equation*}
for some positive constant $C$ independent of $v$. Observing
$$
\lim_{t\to 0}\frac{I_{\la,\e}(te_1)}{t}=-\la\int_{\Omega}\e^{-q(x)}e_{1}\,dx<0,
$$
we can choose $k\in(0,1)$ sufficiently small and set $\|v\|=R :=k(\frac{r+1}{pCl})^\frac{1}{r+1-p}$ such that 
$$
\inf\limits_{\|v\|\leq R}I_{\la,\e}(v)<0.
$$ 
Moreover, since $R<(\frac{r+1}{pCl})^\frac{1}{r+1-p}$, we obtain
\begin{align*}
I_{0,\e}(v)\geq \frac{R^p}{p}-\frac{ClR^{r+1}}{r+1}
:=
2\rho\,(\text{say})>0.
\end{align*}
We define $$\Lambda:=\frac{\rho}{\sup\limits_{\|v\|=R} \left(\displaystyle\frac{1}{1-q(x)}\int_\Om |v|^{1-q(x)}~dx \right)},$$
which is a positive constant and since $\rho,R$ depends on $k,r,p,|\Omega|,C$ so does $\Lambda$. We know that
\begin{equation}\label{knownfact}
(v^{+}+\e)^{1-q(x)}-\e^{1-q(x)}\leq (v^+)^{1-q(x)},
\end{equation}
which gives
\begin{align*}
I_{\lambda,\e}(v)&\geq \frac{1}{p}\int_{\Om}H(\nabla v)^p\,dx-\frac{1}{r+1}\int_{\Om}(v^{+})^{r+1}\,dx-\frac{\la}{1-q(x)}\int_{\Om}(v^{+})^{1-q(x)}\,dx\\
&= I_{0,\e}(v)-\frac{\la}{1-q(x)}\int_{\Om}(v^{+})^{1-q(x)}\,dx.
\end{align*}
Therefore
\begin{align*}
\inf\limits_{\|v\|=R} I_{\la,\e}(v)&\geq\inf\limits_{\|v\|=R}I_{0,\e}(v)-\la \sup\limits_{\|v\|=R} \left(\frac{1}{1-q(x)}\int_\Om |v|^{1-q(x)}~dx \right)\\
&\geq 2\rho -\la \sup\limits_{\|v\|=R} \left(\frac{1}{1-q(x)}\int_\Om |v|^{1-q(x)}~dx \right)\geq \rho,
\end{align*}
if $\la\in(0,\Lambda).$
Lastly, it is easy to see that $I_{0,\e}(te_1) \to -\infty$, as $t\to +\infty$ which implies that we can choose $T>R$ such that $I_{0,\e}(Te_1)<-1$. Hence
\[I_{\la,\e}(Te_1)\leq I_{0,\e}(Te_1)<-1\]
which completes the proof.
\end{proof}

\noi As a consequence of Lemma \ref{MP-geo}, we have
\[\inf\limits_{\|v\|=R}I_{\la,\e}(v) \geq \rho\,\max\{I_{\la,\e}(Te_1), I_{\la,\e}(0)\} = 0.\]
Our next Lemma ensures that $I_{\la,\e}$ satisfies the Palais Smale  $(PS)_c$ condition.

\begin{Proposition}\label{PS-cond}
Let $2\leq p<N$, then $I_{\la,\e}$ satisfies the $(PS)_c$ condition, for any $c \in \mb R$ that is if $\{u_k\}\subset X$ is a sequence satisfying
\begin{equation}\label{PS1}
I_{\la,\e}(u_k)\to c \; \text{and}\; I_{\la,\e}^\prime(u_k) \to 0
\end{equation}
as $k \to \infty$, then $\{u_k\}$ contains a strongly convergent subsequence in $X$.
\end{Proposition}
\begin{proof}
Let $\{u_k\} \subset X$ satisfies \eqref{PS1} then we claim that $\{u_k\}$ must be bounded in $X$. To see this using \eqref{knownfact}, we obtain
\begin{equation}\label{PS2}
\begin{split}
I_{\la,\e}(u_k)- \frac{1}{r+1}I_{\la,\e}^\prime(u_k)u_k &= \left( \frac{1}{p}-\frac{1}{r+1}\right)
\int_{\Omega}H(\nabla u_k)^p\,dx -{\la}\int_\Om \frac{(u_k^+ +\e)^{1-q(x)}-\e^{1-q(x)}}{1-q(x)}~dx\\
& \quad +\frac{\la}{r+1}\int_\Om (u_k^+ +\e)^{-q(x)}u_k~dx\\
& \geq \left( \frac{1}{p}-\frac{1}{r+1}\right)\int_{\Omega}H(\nabla u_k)^p\,dx -\la\int_\Om \frac{(u_k^+)^{1-q(x)}}{1-q(x)}~dx\\
&\quad \quad+ \frac{\la}{r+1}\int_\Om (u_k^+ +\e)^{-q(x)}u_k~dx\\
& \geq \left( \frac{1}{p}-\frac{1}{r+1}\right)\int_{\Omega}H(\nabla u_k)^p\,dx -\la\int_\Om \frac{(u_k^+)^{1-q(x)}}{1-q(x)}~dx-\frac{\la C}{\epsilon(r+1)}\|u_k\|,
\end{split}
\end{equation}
for some positive constant $C$ (independent of $k$) where we have used Lemma \ref{embedding} and the fact $0<q(x)<1$ in $\overline{\Om}$. Due to the same reasoning, we obtain
\begin{equation}\label{PS2-new}
\begin{split}
-\int_\Om \frac{(u_k^+)^{1-q(x)}}{1-q(x)}~dx &\geq -\int_\Om \frac{|u_k|^{1-q(x)}}{1-q(x)}~dx\\
& \geq \frac{-1}{1-\|q\|_{L^\infty(\Om)}}
 \left( \int_{\Om \cap \{|u_k|\geq 1\}}|u_k|^{1-q(x)}dx+ \int_{\Om \cap \{|u_k|< 1\}}|u_k|^{1-q(x)}dx\right) \\
 & \geq \frac{-1}{1-\|q\|_{L^\infty(\Om)}}  \left( \int_{\Om \cap \{|u_k|\geq 1\}}|u_k|dx+ \int_{\Om \cap \{|u_k|< 1\}}|u_k|^{1-\|q\|_{L^\infty(\Om)}}dx\right) \\
& \geq -C\left(\|u_k\|+\|u_k\|^{1-\|q\|_{L^\infty(\Om)}}\right),
\end{split}
\end{equation}
for some positive constant $C$ independent of $k$. Thus inserting \eqref{PS2-new} into \eqref{PS2} and using the fact $r+1>p$ along with Remark \ref{hypormk2}, we get
\begin{equation}\label{PS2-new1}
I_{\la,\e}(u_k)- \frac{1}{r+1}I_{\la,\e}^\prime(u_k)u_k \geq C_1\|u_k\|^p -C\left(\|u_k\|+\|u_k\|^{1-\|q\|_{L^\infty(\Om)}}\right),
\end{equation}
for some positive constant $C_1$ (independent of $k$). Also from \eqref{PS1} it follows that for $k$ large enough
\begin{equation}\label{PS3}
\left| I_{\la,\e}(u_k)- \frac{1}{r+1}I_{\la,\e}^\prime(u_k)u_k\right| \leq C+o(\|u_k\|),
\end{equation}
for some positive constant $C$ (independent of $k$).
Combining \eqref{PS2-new1} and \eqref{PS3}, our claim follows since $p>1$. By reflexivity of $X$, there exists $u_0\in X$ such that up to a subsequence, $u_k \rightharpoonup u_0$ weakly in $X$ as $k \to \infty$.\\
\textbf{Claim:} $u_k \to u_0$ strongly in $X$ as $k \to \infty$.\\
By \eqref{PS1}, we have that
\[\lim_{k\to \infty}\left(\mc \int_{\Om}H(\nabla u_k)^{p-1}\nabla _{\eta}H(\nabla u_k)\nabla u_0\,dx - \la \int_\Om (u_k^+ +\e)^{-q(x)}u_0~dx - \int_\Om (u_k^+)^{r} u_0~dx\right)=0\]
and
\[\lim_{k\to \infty}\left(\mc \int_{\Om}H(\nabla u_k)^{p-1}\nabla _{\eta}H(\nabla u_k)\nabla u_k\,dx - \la \int_\Om (u_k^+ +\e)^{-q(x)}u_k~dx - \int_\Om (u_k^+)^r u_k~dx\right)=0.\]
Now
\begin{equation}\label{PS4}
\begin{split}
&\lim\limits_{k\to\infty}\int_{\Omega}\big(H(\nabla u_k)^{p-1}\nabla_{\eta}H(\nabla u_k)-H(\nabla u_0)^{p-1}\nabla_{\eta}H(\nabla u_0)\big).\nabla(u_k-u_0)\,dx\\
&=\lim\limits_{k\to\infty} \left( \la \int_\Om (u_k^+ +\e)^{-q}u_k~dx + \int_\Om (u_k^+)^r u_k~dx - \la \int_\Om (u_k^+ +\e)^{-q}u_0~dx - \int_\Om (u_k^+)^r u_0~dx\right)\\
&\quad -\lim_{k\to \infty}\left(\int_\Om H(\nabla u_0)^{p-1}\nabla_{\eta}H(\nabla u_0). \nabla u_k~dx - \int_\Om H(\nabla u_0)^p~dx\right).
\end{split}
\end{equation}
From weak convergence of $\{u_k\}$ we get
\begin{equation}\label{PS5}
\lim_{k\to \infty}\left(\int_\Om H(\nabla u_0)^{p-1}\nabla_{\eta}H(\nabla u_0)\nabla u_k~dx - \int_\Om H(\nabla u_0)^p~dx\right)=0.
\end{equation}
Also, we have
\begin{align*}
\left|(u_k^++\e)^{-q(x)}u_0\right| \leq\e^{-q(x)}u_0\text{ and }
\int_\Om \left|\e^{-q(x)}u_0\right|dx \leq \|\e^{-q(x)}\|_{L^\infty(\Om)}\int_\Om|u_0|~dx< +\infty.
\end{align*}
Thus Lebesgue Dominated convergence theorem gives that
\begin{equation}\label{PS6}
\lim_{k \to \infty} \int_\Om (u_k^+ +\e)^{-q(x)}u_0~dx = \int_\Om (u_0^+ +\e)^{-q(x)}u_0~dx.
\end{equation}
 Since $u_k \to u_0$ pointwise in $\Om$, for any measurable subset $E$ of $\Om$ we have
\begin{equation*}
\begin{split}
\int_E |(u_k^++\e)^{-q(x)}u_k |~dx&\leq \int_E\e^{-q(x)}u_k~dx\leq\|\e^{-q(x)}\|_{L^\infty(\Om)}\|u_k\|_{L^{p^*}(\Om)}|E|^{\frac{p^*-1}{p^*}}\leq C(\e)|E|^{\frac{p^*-1}{p^*}},
\end{split}
\end{equation*}
so from Vitali convergence theorem, it follows that
\begin{equation}\label{PS7}
\lim\limits_{k\to\infty} \la \int_\Om (u_k^+ +\e)^{-q(x)}u_k~dx = \la \int_\Om (u_0^+ +\e)^{-q(x)}u_0~dx.
\end{equation}
Similarly, due to $r+1<p^*$, we have
\[\int_E |(u_k^+)^ru_0|~dx \leq \|u_0\|_{L^{p^*}(\Om)} \left(\int_E (u_k^+)^{rp^*{'}}~dx\right)^{\frac{1}{p^*{'}}}\leq C_3 |E|^{\alpha}  \]
and
\[\int_E |(u_k^+)^ru_k|~dx \leq \|u_k\|_{L^{p^*}(\Om)} \left(\int_E (u_k^+)^{rp^*{'}}~dx\right)^{\frac{1}{p*{'}}}\leq C_4 |E|^{\beta}  \]
for some positive constants $C_3,C_4,\alpha$ and $\beta$. Therefore Vitali convergence theorem gives
\begin{equation}\label{PS8}
\lim_{k \to \infty} \int_\Om (u_k^+)^ru_0~dx  =\int_\Om (u_0^+)^ru_0~dx,
\end{equation}
and
\begin{equation}\label{PS9}
\lim_{k \to \infty} \int_\Om (u_k^+)^ru_k~dx  =\int_\Om (u_0^+)^ru_0~dx.
\end{equation}
Using \eqref{PS5}, \eqref{PS6}, \eqref{PS7}, \eqref{PS8} and \eqref{PS9} in \eqref{PS4}, we obtain
\[\lim\limits_{k\to\infty}\int_{\Omega}\big(H(\nabla u_k)^{p-1}\nabla _{\eta}H(\nabla u_k)-H(\nabla u_0)^{p-1}\nabla _{\eta}H(\nabla u_0)\big).\nabla(u_k-u_0)\,dx =0.\]
Then using \eqref{dualthmapp3} from Corollary \ref{dualthmcor}, we obtain $u_k\to u_0$ strongly in $X$ as $k\to\infty$ which proves our claim. 
\end{proof}
\begin{Remark}\label{multrmk}
From Lemma \ref{MP-geo}, Proposition \ref{PS-cond} and Mountain Pass Lemma, for every $\la\in(0,\Lambda)$, there exists a $\zeta_\e \in X$ such that $I_{\lambda,\e}^\prime(\zeta_\e)=0$ {and}
$$
I_{\lambda,\e}(\zeta_{\e})=\inf_{\gamma\in\Gamma}\max_{t \in [0,1]}I_{\la,\e}(\gamma (t)) \geq \rho >0,
$$
where 
$$
\Gamma =\big\{\gamma \in C([0,1],X):\gamma(0)=0, \gamma(1)=Te_1\big\}.
$$
Using \eqref{PS2-new} together with Vitali convergence theorem, if $u_k\rightharpoonup u_0$ weakly in $X$, then we have
$$
\lim_{k\to\infty}\int_{\Om}\frac{(u_k+\e)^{1-q(x)}-\e^{1-q(x)}}{1-q(x)}\,dx=\int_{\Om}\frac{(u_0+\e)^{1-q(x)}-\e^{1-q(x)}}{1-q(x)}\,dx.
$$
Therefore, $I_{\la,\e}$ is weakly lower semicontinuous. Furthermore, as a consequence of Lemma \ref{MP-geo}, since for every $\la\in(0,\Lambda)$ we have $\inf\limits_{\|v\|\leq R} I_{\la,\e}(v)<0$, there exists a nonzero $\nu_\e\in X$ such that $\|\nu_\e\| \leq R$ and
\begin{equation}\label{limit-pass}
\inf\limits_{\|v\|\leq R} I_{\la,\e}(v) =I_{\la,\e}(\nu_\e)<0<\rho \leq I_{\la,\e}(\zeta_\e).
\end{equation}
Thus, $\zeta_\e$ and $\nu_\e$ are two different non trivial critical points of $I_{\la,\e}$, provided $\la\in(0,\Lambda)$.
\end{Remark}

\begin{Lemma}\label{non-negative}
Let $2\leq p<N$, then the critical points $\zeta_\e$ and $\nu_\e$ of $I_{\la,\e}$ are nonnegative in $\Omega.$
\end{Lemma}
\begin{proof}
Testing $(P_{\la,\e})$ with $\min\{\zeta_\e,0\}$ and $\min\{\nu_\e,0\}$, proceeding similarly as in the proof of Lemma \ref{auxresult}, we have $\zeta_\e,\nu_\e\geq 0$ in $\Om$.
\end{proof}
\begin{Remark}\label{nonnegrmk}
As in Remark \ref{sinrmk}, if $\Delta_{H,p}=\Delta_p$ or $\mathcal{S}_p$ as given by \eqref{ex}, then noting Lemma \ref{AI}, we have Proposition \ref{PS-cond} and Lemma \ref{non-negative} valid for any $1<p<N$.
\end{Remark}
\begin{Lemma}\label{apriori}
There exists a constant $\Theta>0$ (independent of $\e$) such that $\|v_\e\| \leq \Theta$, where $v_\e = \zeta_\e$ or $\nu_\e$.
\end{Lemma}
\begin{proof}
The result trivially holds if $v_\e = \nu_\e$ so we deal with the case $v_\e= \zeta_\e$. Recalling the terms from Lemma \ref{MP-geo}, we define $A = \max\limits_{t \in [0,1]}I_{0,\e}(tTe_1)$ then
\[A \geq \max_{t \in [0,1]} I_{\la,\e}(tTe_1) \geq\inf_{\gamma\in\Gamma}\max_{t \in [0,1]}I_{\la,\e}(\gamma (t)) = I_{\la,\e}(\zeta_\e)\geq \rho>0>I_{\la,\e}(\nu_\e).\]
Therefore
\begin{equation}\label{ap1}
\frac{1}{p}\int_{\Om}H(\nabla\zeta_\e)^p\,dx-{\la}\int_\Om \frac{(\zeta_\e +\e)^{1-q(x)}-\e^{1-q(x)}}{1-q(x)}~dx -\frac{1}{r+1}\int_\Om \zeta_\e^{r+1}~dx \leq A.
\end{equation}
	Choosing $\phi=-\frac{\zeta_\e}{p+1}$ as a test function in $(P_{\la,\e})$ we obtain
\begin{equation}\label{ap2}
-\frac{1}{r+1}\int_{\Om}H(\nabla\zeta_\e)^p\,dx+\frac{\la}{r+1}\int_{\Om}\frac{\zeta_\e}{(\zeta_\e+\e)^{q(x)}}\,dx+\frac{1}{r+1}\int_{\Om}\zeta_\e^{r+1}\,dx=0.
\end{equation}
Adding \eqref{ap1} and \eqref{ap2} we get
\begin{align*}
\left(\frac{1}{p}-\frac{1}{r+1}\right)\int_{\Om}H(\nabla\zeta_\e)^p\,dx
 &\leq {\la}\int_\Om \frac{(\zeta_\e +\e)^{1-q(x)}-\e^{1-q(x)}}{1-q(x)}~dx -\frac{\la}{r+1}\int_{\Om}\frac{\zeta_\e}{(\zeta_\e+\e)^{q(x)}}\,dx+A\\
 & \leq {\la}\int_\Om \frac{(\zeta_\e +\e)^{1-q(x)}-\e^{1-q(x)}}{1-q(x)}~dx +A\\
 & \leq C(\|\zeta_\e\|+\|\zeta_\e\|^{1-\|q\|_{L^\infty(\Om)}})+A,
\end{align*}
for some positive constant $C$ being independent of $\e$, where the last inequality is deduced using the estimate \eqref{PS2-new}, H\"older inequality along Lemma \ref{embedding}. Therefore, using $r+1>p$, the sequence $\{\zeta_\e\}$ is uniformly bounded in $X$ with respect to $\e$. This completes the proof.
\end{proof}

\section{Proof of the main results}
\textbf{Proof of Theorem \ref{thm1}:}
\begin{enumerate}
\item[(a)]\textbf{Uniqueness:}
Let $u,v\in X$ be two distinct weak solutions of $\eqref{maineqn}$. Then by Lemma \ref{testfn} we can choose $\phi=(u-v)^{+}$ as a test function in \eqref{psingtest} to get
$$
\langle-\Delta_{H,p} u + \Delta_{H,p} v,(u-v)^{+}\rangle=\int_{\Omega}f\left(\frac{1}{u^{q(x)}}-\frac{1}{v^{q(x)}}\right)(u-v)^{+}\,dx\leq 0.
$$
Using Lemma \ref{alg} we obtain $u\leq v$ in $\Omega$. Similarly, choosing $\phi=(v-u)^{+}$ as a test function in \eqref{psingtest} we get $v\leq u$ in $\Omega$. Hence the uniqueness follows.
\item[(b)]\textbf{Existence:} Let $u_n\in X$ denotes the weak solution of \eqref{mainapprox} then choosing $u_n$ as a test function in \eqref{mainapprox} and using Remark \ref{hypormk2}, we get
$$
\int_{\Om}H(\nabla u_n)^p\,dx\leq \int_{\Om}\frac{f_n(x)u_n}{\big(u_n+\frac{1}{n}\big)^{q(x)}}\,dx.
$$
Now denote by $\omega_{\delta}=\Omega\setminus\overline{\Omega_{\delta}}$ then by Lemma \ref{approxexi}, we get $u_n\geq c_{\omega_{\delta}}>0$ in $\omega_{\delta}$. We observe that
\begin{align*}
\int_{\Om}\frac{f_n(x)u_n}{(u_n+\frac{1}{n})^{q(x)}}\,dx
&=\int_{\overline{{\Om}_{\delta}}}\frac{f_n(x)u_n}{(u_n+\frac{1}{n})^{q(x)}}\,dx+\int_{\omega_{\delta}}\frac{f_n(x)u_n}{(u_n+\frac{1}{n})^{q(x)}}\,dx\\
&\leq \int_{\overline{{\Om}_{\delta}}}f(x)u_n^{1-q(x)}\,dx+ \int_{\omega_{\delta}}\frac{f(x)}{u_n^{q(x)}}u_n\,dx\\
&\leq \int_{\overline{{\Om}_{\delta}}\cap\{u_n\leq 1\}}f(x)\,dx+\int_{{\overline{{\Om}_{\delta}}}\cap\{{u_n\geq 1\}}}f(x)u_n\,dx+\int_{\omega_{\delta}}\frac{f(x)}{c_{\omega_{\delta}}^{q(x)}}u_n\,dx\\
&\leq \|f\|_{L^1(\Om)}+\big(1+\|c_{\omega_{\delta}}^{-q(x)}\|_{L^\infty(\Om)}\big)\int_{\Om}f(x)u_n\,dx.
\end{align*}
Let $2\leq p<N$. Then, using H$\ddot{\text{o}}$lder's inequality and Lemma \ref{embedding} we obtain for $f\in L^m(\Om)$ with $m=(p^{*})'$
$$
\|u_n\|^{p}\leq \|f\|_{L^1(\Om)}+C\|f\|_{L^m(\Om)}\|u_n\|,
$$
which implies that the sequence $\{u_n\}$ is uniformly bounded in $X$. Thus up to a subsequence, by Lemma \ref{embedding} for $1\leq t<p^{*}$, we get 
\begin{align*}
u_n&\rightharpoonup u\text{ weakly in } X,\\
\,u_n&\to u\text{ strongly in } L^t(\Om),
\text{ and }\\
u_n&\to u\text{ pointwise in }\Om\text{ as }n \to \infty.
\end{align*}
Let $\phi\in C_c^{\infty}(\Omega)$ be such that $\mathrm{supp}\,\phi=\omega$ where $\omega \Subset \Om$. Then by Lemma \ref{approxexi}, there exists a constant $c_{\omega}>0$ independent of $n$ such that $u_n\geq c_{\omega}>0$ in $\omega$. Hence, we get
$$
\frac{f_n}{\left(u_n+\frac{1}{n}\right)^{q(x)}}\phi \leq c_{\omega}^{-q(x)}f(x)\phi(x)\in L^1(\Om).
$$
Thus we can apply Boccardo-Murat \cite[Theorem 2.1]{BocMu} to obtain that upto a subsequence still denoted by $\{u_n\}$, we have
$$
\lim_{n\to\infty}\nabla u_n(x)=\nabla u(x)\text{ for }x\text{ in }\Omega.
$$
Hence, we have 
$$
\lim_{n\to\infty}H(\nabla u_n(x))^{p-1}\nabla_{\eta}H\big(\nabla u_n(x)\big)=H(\nabla u(x))^{p-1}\nabla_{\eta}H\big(\nabla u(x)\big)\text{ for }x\text{ in }\Omega.
$$
Moreover, by Corollary \ref{regrmk} we observe that
$$
\|H(\nabla u_n)^{p-1}\nabla_{\eta}H(\nabla u_n)\|_{L^\frac{p}{p-1}(\Omega)}^\frac{p}{p-1}\leq\|u_n\|^p\leq C,
$$
for some positive constant $C$ which is independent of $n$.
As a consequence, the sequence
$$
H(\nabla u_n)^{p-1}\nabla_{\eta}H(\nabla u_n)\rightharpoonup H(\nabla u)^{p-1}\nabla_{\eta}H(\nabla u)\text{ weakly in }X \text{ as } n\to\infty.
$$
Indeed, since the weak limit is independent of the subsequence, the above fact holds for every $n$.
Therefore, we have
\begin{equation}\label{limthm1}
\lim_{n\to\infty}\int_{\Omega}H(\nabla u_n(x))^{p-1}\nabla_{\eta}H(\nabla u_n(x))\nabla\phi\,dx=\int_{\Omega}H(\nabla u)^{p-1}\nabla_{\eta}H(\nabla u)\nabla\phi\,dx,
\end{equation}
for every $\phi\in C_c^{\infty}(\Omega)$. Moreover by the Lebsegue dominated convergence theorem, we have
\begin{equation}\label{limthm2}
\lim_{n\to\infty}\int_{\Omega}\frac{f_n}{\big(u_n+\frac{1}{n}\big)^{q(x)}}\phi\,dx=\int_{\Omega}\frac{f}{u^{q(x)}}\,\phi\,dx\text{ for all }\phi\in C_c^{\infty}(\Omega).
\end{equation}
Now, we define by 
$$
u:=\lim_{n\to\infty}u_n.
$$
Then, due to the monotonicity of $\{u_n\}$ we get $u\geq c_{\omega}>0$ for every $\omega\Subset\Omega$. Hence from \eqref{limthm1} and \eqref{limthm2}, we get $u$ is our required solution and the result follows.

Noting Lemma \ref{embedding}, for $p\geq N$ the result follows similarly.
\end{enumerate}
\noi \textbf{Proof of Theorem \ref{thm1new}:} Noting Lemma \ref{AI} and Remark \ref{sinrmk}, proceeding similarly as in the proof of Theorem \ref{thm1}, the result follows.

\noi \textbf{Proof of Theorem \ref{thm3}:}
As a resultant of Lemma \ref{non-negative} and Lemma \ref{apriori}, upto a subsequence we get that $\zeta_\e \rightharpoonup \zeta_0$ and $\nu_\e \rightharpoonup \nu_0$ weakly in $X$ as $\e \to 0^+$, for some non negative $\zeta_0,\nu_0\in X$.  In the sequel, we establish that $\zeta_0\neq \nu_0$ and forms a weak solution to our problem \eqref{maineqn1}. For convenience, we denote by $v_0$ either $\zeta_0$ or $\nu_0$. We prove the result in the following steps.\\
\textbf{Step $1$.} Here, we prove that $v_0\in X$ is a weak solution to the problem \eqref{maineqn1}.\\
We observe that for any $\e\in(0,1)$ and $t\geq 0$,
$$
\frac{\la}{(t+\e)^{q(x)}}+t^r\geq \frac{\la}{(t+1)^{q(x)}}+t^r\geq \text{min}\left\{1,\frac{\la}{2}\right\}:=C>0, \text{ say}.
$$
As a consequence we get
$$
-\Delta_{H,p} \,v_\e=\frac{\la}{(v_\e+\e)^{q(x)}}+v_\e^r\geq C>0.
$$
By Lemma \ref{auxresult}, let $\xi\in X$ satisfies 
$$
-\Delta_{H,p}\xi=C\text{ in }\Omega,\,\xi>0\text{ in }\Omega,
$$
such that for every $\omega\Subset\Om$, there exists a constant $c_{\omega}>0$ satisfying $\xi\geq c_{\omega}>0$ in $\Om$. 
Then, for every nonnegative $\phi\in X$, we have
\begin{align*}
\int_{\Om}H(v_\e)^{p-1}\nabla_{\eta}H(v_\e)\nabla\phi\,dx&=\int_{\Om}\Big(\frac{\la}{(v_\e+\e)^{q(x)}}+v_\e^r\Big)\phi\,dx\\
&\geq\int_{\Om}C\phi\,dx=\int_{\Om}H(\nabla\xi)^{p-1}\nabla_{\eta}H(\nabla\xi)\nabla\phi\,dx.
\end{align*}
Now choosing $\phi=(\xi-v_\e)^+$, we obtain
$$
\int_{\Om}\Big\{H(\nabla\xi)^{p-1}\nabla_{\eta}H(\nabla\xi)-H(\nabla v_\e)^{p-1}\nabla_{\eta}H(\nabla v_\e)\Big\}\nabla(\xi-v_\e)^+\,dx\leq 0.
$$
Now applying Lemma \ref{alg}, we obtain $v_\e\geq \xi$ in $\Om$. Hence, we get the existence of a constant $c_{\omega}>0$ (independent of $\e$) such that
\begin{equation}\label{uniform}
v_\e\geq c_\omega>0,\text{ for every }\omega\Subset\Om. 
\end{equation}
Therefore using $r+1<p^*$ along with Lemma \ref{apriori} and the fact \eqref{uniform} as in the proof of Theorem \ref{thm1} we can apply Boccardo-Murat \cite[Theorem 2.1]{BocMu} to obtain
$$
\int_{\Om}H(\nabla v_0)^{p-1}\nabla_{\eta}H(\nabla v_0)\nabla \phi\,dx=\la\int_{\Om}\frac{\phi}{v_0^q(x)}\,dx+\int_{\Om}v_0^{r}\phi\,dx.
$$
Hence the claim follows.\\
\textbf{Step $2$.} Now we are going to prove that $\zeta_0\neq \nu_0$.\\
Choosing $\phi=v_\e\in X$ as a test function in $(P_{\la,\e})$ we get
$$
\int_{\Om}H(\nabla v_\e)^{p-1}\nabla_{\eta}H(\nabla v_\e)\nabla\phi\,dx=\la\int_{\Om}\frac{v_\e}{(v_\e+\e)^{q(x)}}\,dx+\int_{\Om}v_\e^{r+1}\,dx.
$$
Since $r+1<p^{*}$, using Lemma \ref{embedding} we obtain
$$
\lim\limits_{\e\to 0^+}\int_{\Om}(v_\e)^{r+1}\,dx=\int_{\Om}v_0^{r+1}\,dx.
$$
Moreover, since
$$
0\leq \frac{v_\e}{(v_\e+\e)^{q(x)}}\leq v_\e^{1-q(x)},
$$
using \eqref{PS2-new} together with Vitali convergence theorem, we get
$$
\la\lim\limits_{\e\to 0^+}\int_{\Om}\frac{v_\e}{(v_\e+\e)^{q(x)}}\,dx=\la\int_{\Om}v_0^{1-q(x)}\,dx.
$$
Therefore for every $\phi\in X$, we have
$$
\lim_{\e\to 0^+}\int_{\Om}H(\nabla v_\e)^{p-1}\nabla_{\eta}H(\nabla v_\e)\nabla\phi\,dx=\la\int_{\Om}v_0^{1-q(x)}\,dx+\int_{\Om}v_0^{r+1}\,dx.
$$
Using Lemma \ref{testfn} we can choose $\phi=v_0$ as  a test function in \eqref{pursingtest} to deduce that
$$
\int_{\Om}H(\nabla v_0)^p\,dx=\la\int_{\Om}v_0^{1-q(x)}\,dx+\int_{\Om}v_0^{r+1}\,dx.
$$
Hence we obtain
$$
\lim\limits_{\e\to 0}\int_{\Om}H(\nabla v_\e)^p\,dx=\int_{\Om}H(\nabla v_0)^p\,dx.
$$
By the Vitali convergence theorem, we get
$$
\lim\limits_{\e\to 0}\int_{\Om}[(v_\e+\e)^{1-q(x)}-\e^{1-q(x)}]\,dx=\int_{\Om}v_0^{1-q(x)}\,dx,
$$
which together with the strong convergence of $v_\e$ in $X$ implies
$
\lim\limits_{\e\to 0}I_{\la,\e}(v_\e)=I_{\la}(v_0).
$
Hence from \eqref{limit-pass} we get $\zeta_0\neq \nu_0.$ 

\noi \textbf{Proof of Theorem \ref{thm3new}:} Noting Lemma \ref{AI} and Remark \ref{nonnegrmk}, proceeding similarly as in the proof of Theorem \ref{thm3}, the result follows.

 Kaushik Bal\\
Department of Mathematics and Statistics,\\
Indian Institute of Technology Kanpur,\\
Kanpur-208016,\\
Uttar Pradesh, India\\
Email: kaushik@iitk.ac.in\\

 Prashanta Garain\\
Department of Mathematics,\\
Ben-Gurion University of the Negev,\\
P.O.B.-653,\\
Beer Sheva-8410501, Israel\\
Email: pgarain92@gmail.com\\

 Tuhina Mukherjee\\
Department of Mathematics,\\
Indian Institute of Technology Jodhpur,\\
Jodhpur-342037,\\
Rajasthan, India\\
Email: tulimukh@gmail.com, tuhina@iitj.ac.in

\end{document}